\documentclass[10pt]{amsart}
\usepackage{amssymb}
\usepackage{dsfont}
\usepackage{amscd}
\usepackage{mathabx}
\usepackage[mathscr]{euscript} 
\usepackage[all]{xy}
\usepackage{url}
\usepackage{stmaryrd}
\usepackage{comment}
\usepackage[retainorgcmds]{IEEEtrantools}
\usepackage{hyperref}
\title{Galois actions of finitely generated groups rarely have model companions}
\author[\"{O}. BEYARSLAN]{\"{O}zlem Beyarslan$^{\clubsuit}$}
\thanks{$^{\clubsuit}$ Supported by the T\"{u}bitak 1001 grant no. 119F397.}
\address{$^{\clubsuit}$Bo\v{g}azi\c{c}i \"{U}niversitesi}
\email{ozlem.beyarslan@boun.edu.tr}
\author[P. KOWALSKI]{Piotr Kowalski$^{\spadesuit}$}
\thanks{$^{\spadesuit}$ Supported by the Narodowe Centrum Nauki grant no. 2018/31/B/ST1/00357 and by the T\"{u}bitak 1001 grant no. 119F397.}
\address{$^{\spadesuit}$Instytut Matematyczny\\
Uniwersytet Wroc{\l}awski\\
Wroc{\l}aw\\
Poland}
\email{pkowa@math.uni.wroc.pl} \urladdr{http://www.math.uni.wroc.pl/\textasciitilde pkowa/ }

 \DeclareMathOperator{\aut}{Aut} \DeclareMathOperator{\id}{id}

  \DeclareMathOperator{\gal}{Gal}

 \DeclareMathOperator{\alg}{alg}

\DeclareMathOperator{\li}{\underleftarrow{\lim}}

\DeclareMathOperator{\rat}{rat}
\DeclareMathOperator{\sep}{sep}

\DeclareMathOperator{\dcf}{DCF}

\newtheorem{theorem}{Theorem}[section]
\newtheorem{prop}[theorem]{Proposition}
\newtheorem{lemma}[theorem]{Lemma}
\newtheorem{cor}[theorem]{Corollary}
\newtheorem{fact}[theorem]{Fact}

\theoremstyle{definition}
\newtheorem{definition}[theorem]{Definition}
\newtheorem{example}[theorem]{Example}
\newtheorem{remark}[theorem]{Remark}
\newtheorem{question}[theorem]{Question}

\begin{document} 
\newcommand{\lili}{\underleftarrow{\lim }}
\newcommand{\coco}{\underrightarrow{\lim }}
\newcommand{\twoc}[3]{ {#1} \choose {{#2}|{#3}}}
\newcommand{\thrc}[4]{ {#1} \choose {{#2}|{#3}|{#4}}}
\newcommand{\Zz}{{\mathds{Z}}}
\newcommand{\Ff}{{\mathds{F}}}
\newcommand{\Cc}{{\mathds{C}}}
\newcommand{\Rr}{{\mathds{R}}}
\newcommand{\Nn}{{\mathds{N}}}
\newcommand{\Qq}{{\mathds{Q}}}
\newcommand{\Kk}{{\mathds{K}}}
\newcommand{\Pp}{{\mathds{P}}}
\newcommand{\ddd}{\mathrm{d}}
\newcommand{\Aa}{\mathds{A}}
\newcommand{\dlog}{\mathrm{ld}}
\newcommand{\ga}{\mathbb{G}_{\rm{a}}}
\newcommand{\gm}{\mathbb{G}_{\rm{m}}}
\newcommand{\gaf}{\widehat{\mathbb{G}}_{\rm{a}}}
\newcommand{\gmf}{\widehat{\mathbb{G}}_{\rm{m}}}
\newcommand{\ka}{{\bf k}}
\newcommand{\ot}{\otimes}
\newcommand{\si}{\mbox{$\sigma$}}
\newcommand{\ks}{\mbox{$({\bf k},\sigma)$}}
\newcommand{\kg}{\mbox{${\bf k}[G]$}}
\newcommand{\ksg}{\mbox{$({\bf k}[G],\sigma)$}}
\newcommand{\ksgs}{\mbox{${\bf k}[G,\sigma_G]$}}
\newcommand{\cks}{\mbox{$\mathrm{Mod}_{({A},\sigma_A)}$}}
\newcommand{\ckg}{\mbox{$\mathrm{Mod}_{{\bf k}[G]}$}}
\newcommand{\cksg}{\mbox{$\mathrm{Mod}_{({A}[G],\sigma_A)}$}}
\newcommand{\cksgs}{\mbox{$\mathrm{Mod}_{({A}[G],\sigma_G)}$}}
\newcommand{\crats}{\mbox{$\mathrm{Mod}^{\rat}_{(\mathbf{G},\sigma_{\mathbf{G}})}$}}
\newcommand{\crat}{\mbox{$\mathrm{Mod}^{\rat}_{\mathbf{G}}$}}
\newcommand{\cratinv}{\mbox{$\mathrm{Mod}^{\rat}_{\mathbb{G}}$}}
\newcommand{\ra}{\longrightarrow}
\newcommand{\bdcf}{B-\dcf}
\makeatletter
\providecommand*{\cupdot}{%
  \mathbin{%
    \mathpalette\@cupdot{}%
  }%
}
\newcommand*{\@cupdot}[2]{%
  \ooalign{%
    $\m@th#1\cup$\cr
    \sbox0{$#1\cup$}%
    \dimen@=\ht0 %
    \sbox0{$\m@th#1\cdot$}%
    \advance\dimen@ by -\ht0 %
    \dimen@=.5\dimen@
    \hidewidth\raise\dimen@\box0\hidewidth
  }%
}

\providecommand*{\bigcupdot}{%
  \mathop{%
    \vphantom{\bigcup}%
    \mathpalette\@bigcupdot{}%
  }%
}
\newcommand*{\@bigcupdot}[2]{%
  \ooalign{%
    $\m@th#1\bigcup$\cr
    \sbox0{$#1\bigcup$}%
    \dimen@=\ht0 %
    \advance\dimen@ by -\dp0 %
    \sbox0{\scalebox{2}{$\m@th#1\cdot$}}%
    \advance\dimen@ by -\ht0 %
    \dimen@=.5\dimen@
    \hidewidth\raise\dimen@\box0\hidewidth
  }%
}
\makeatother

\def\Ind#1#2{#1\setbox0=\hbox{$#1x$}\kern\wd0\hbox to 0pt{\hss$#1\mid$\hss}
\lower.9\ht0\hbox to 0pt{\hss$#1\smile$\hss}\kern\wd0}

\def\ind{\mathop{\mathpalette\Ind{}}}

\def\notind#1#2{#1\setbox0=\hbox{$#1x$}\kern\wd0
\hbox to 0pt{\mathchardef\nn=12854\hss$#1\nn$\kern1.4\wd0\hss}
\hbox to 0pt{\hss$#1\mid$\hss}\lower.9\ht0 \hbox to 0pt{\hss$#1\smile$\hss}\kern\wd0}

\def\nind{\mathop{\mathpalette\notind{}}}


\maketitle

\begin{abstract}
We show that if $G$ is a finitely generated group such that its profinite completion $\widehat{G}$ is ``far from being projective'' (that is the kernel of the universal Frattini cover of $\widehat{G}$ is not a small profinite group), then the class of existentially closed $G$-actions on fields is not elementary. Since any infinite, finitely generated, virtually free, and not free group is ``far from being projective'', the main result of this paper corrects an error in our paper \emph{Model theory of fields with virtually free group actions}, Proc. London Math. Soc., (2) 118 (2019), 221--256 by showing the negation of Theorem 3.26 in that paper.
\end{abstract}

\section{Introduction}
The aim of this paper is to correct an error which appeared in our paper \cite{BK}. Unfortunately, this error is extremely serious: in short, in the current paper we prove the \emph{negation} of \cite[Theorem 3.26]{BK}. In fact, our main result (Theorem \ref{main}) implies a ``strong negation'' of \cite[Theorem 3.26]{BK}, in the sense which is explained below. The statement \cite[Theorem 3.26]{BK} says that if $G$ is a finitely generated and virtually free group, then the theory of actions of $G$ on fields has a model companion. This result had been previously known in the cases when $G$ is free or finite. Theorem \ref{main} implies that if $G$ is finitely generated, infinite, virtually free, and not free, then the theory of actions of $G$ on fields does \emph{not} have a model companion. Therefore, Theorem \ref{main} (more precisely: Corollary \ref{strneg}) can be considered as a ``strongest possible'' negation of \cite[Theorem 3.26]{BK}.

We would like to describe briefly here the content of \cite{BK} as we see it now \emph{after} realizing our mistake. There are three main statements in \cite{BK}:
\begin{enumerate}
  \item a statement about companionability of actions of virtually free groups on fields (\cite[Theorem 3.26]{BK}),
  \item computations of certain profinite groups (\cite[Theorem 4.6]{BK}),
  \item non-companionability of actions of $\Zz\rtimes \Zz$ on fields (\cite[Corollary 5.7]{BK}).
\end{enumerate}
The first statement is false (see Corollary \ref{strneg}). We use the second result above in the current paper to show the negation of \cite[Theorem 3.26]{BK}. We generalize the third result above in the current paper to the case of nilpotent groups (see Corollary \ref{nilpans}).

The mathematical reason for the error we made in \cite{BK} can be explained in very simple terms: a tensor product of domains need not be a domain (e.g. $\Cc\otimes_{\Rr}\Cc\cong \Cc\times \Cc$)! In the proof of \cite[Theorem 3.26]{BK}, we implicitly (and incorrectly) assumed that a fiber product of $K$-irreducible algebraic varieties is again $K$-irreducible, which need not be true even when $K$ is algebraically closed. This error has its roots already in the Introduction to \cite{BK}, where we write: ``Our prolongation process may be seen as a constructive explanation ...'', however the constructive procedure described in \cite{BK} gives only a canonical difference \emph{ring} extension, where the bigger ring need not be a domain. We can get a difference quotient which is a domain, but finding this quotient corresponds to finding an appropriate difference prime ideal and this is not constructive, since it requires Zorn's Lemma.

In this paper, we show that the class of existentially closed $G$-fields, for some finitely generated groups $G$, is not elementary by showing that such $G$-fields are not bounded (as fields), but they still have absolute Galois groups of bounded cardinality. Since a first-order theory cannot axiomatize any class of infinite structures of bounded cardinality, assuming that a model companion exists leads to a contradiction. More precisely, we show (Theorem \ref{main}) that if $G$ is finitely generated and the profinite completion $\widehat{G}$ is ``far from being projective'' (meaning that the kernel of the universal Frattini cover of $\widehat{G}$ is not small, see Definition \ref{defkg}), then the theory of actions of $G$ on fields has no model companion. It was shown in \cite{BK} that infinite, finitely generated, virtually free, and not free groups are ``far from being projective'', so Theorem \ref{main} implies that for such groups $G$, the theory of actions of $G$ on fields has no model companion. However, Theorem \ref{main} covers many more cases, for example the group $\Zz\times \Zz$ is ``far from being projective'', so, as a special case, we also give a new proof of a rather mysterious Hrushovski's result about non-companionability of the theory of partial difference fields, that is fields with two commuting automorphisms.

In this paper, we give a counterexample (quite an unexpected one) to \cite[Conjecture 5.9]{BK}. We also give a counterexample (see Remark \ref{remque}(2)) to a conjecture of Hoffmann (\cite[Conjecture 5.2]{Hoff3}) and to a conjecture about relations between the free product of groups and companionability of the corresponding theories of their actions on fields (see Remark \ref{remque}(3)). Using \cite[Theorem 1]{ols}, we also give a counterexample to \cite[Theorem 6]{Sjo} (see Remark \ref{remabs}(3)).

The paper is organized as follows. In Section \ref{sec2}, we collect the necessary definitions and results which are needed for the sequel. In Section \ref{sec3}, we show our main non-companionability results. In Section \ref{sec4}, we deal with the nilpotent and the free product cases and we also summarize what we know about the companionability of the theories of group actions on fields.

We would like to thank Alexander Ivanov for fruitful discussions about geometric group theory and Alexander Ol'shanskii for pointing out to us the example of the group $G$ which appears in Remark \ref{remabs}(3).

\section{Absolute Galois groups of $G$-fields}\label{sec2}
In this section, we set our notation and present the necessary notions and results.
Let $G$ be an arbitrary group.  By
$$\widehat{G}:=\underset{H\trianglelefteqslant_f G}{\li }G/H,$$
where $H$ ranges over normal subgroups of $G$ of finite index, we denote the \emph{profinite completion} of $G$ considered as a profinite topological group.

We recall (see \cite[Definition 22.5.1]{FrJa}) that a continuous epimorphism of profinite groups $f:\mathcal{G}\to \mathcal{H}$ is a \emph{Frattini cover}, if for any closed subgroup $\mathcal{G}_0\leqslant \mathcal{G}$, we have that $\mathcal{G}_0=\mathcal{G}$ if and only if $f(\mathcal{G}_0)=\mathcal{H}$. A Frattini cover $f:\mathcal{G}\to \mathcal{H}$ is called a \emph{universal Frattini cover}, if the profinite group $\mathcal{G}$ is \emph{projective} (see \cite[Proposition 22.6.1]{FrJa}). A universal Frattini cover $f:\mathcal{G}\to \mathcal{H}$ is unique up to a topological isomorphism over $\mathcal{H}$ and if $f:\mathcal{G}\to \mathcal{H}$ is a universal Frattini cover, then we denote its domain $\mathcal{G}$ by $\widetilde{\mathcal{H}}$. A profinite group $\mathcal{G}$ is \emph{small}, if for any $n>0$, there are only finitely many open normal subgroups of $\mathcal{G}$ of index $n$. If $\mathcal{G}$ is a profinite group, then the \emph{rank} of $\mathcal{G}$ (denoted $\mathrm{rk}(\mathcal{G})$) is the smallest cardinal $\kappa$ such that there is $A\subseteq \mathcal{G}$ of cardinality $\kappa$ such that $A$ converges to $1\in \mathcal{G}$ and the subgroup generated by $A$ is dense in $\mathcal{G}$ (see \cite[Section 17.1]{FrJa}).

In Section \ref{sec4}, we will use  the following notions and facts from the theory of profinite groups (proofs can be found in \cite[Chapter 22.9]{FrJa}). The classical Sylow theory for finite groups generalizes to the profinite case after replacing the notion of a $p$-subgroup with the notion of a closed pro-$p$ subgroup. In particular, for a prime $p$ and a profinite group $\mathcal{G}$, $p$-Sylow subgroups of $\mathcal{G}$ exist and they are conjugate. We also have the corresponding results about pronilpotent groups: a profinite group is pronilpotent if and only if it is the product of its unique $p$-Sylow subgroups. If $\mathcal{G}$ is a pronilpotent group and $p$ is a prime number, then we denote by $\mathcal{G}_p$ the unique $p$-Sylow subgroup of $\mathcal{G}$.  By ``cl'', we denote the topological closure (in an ambient profinite group). For a prime number $p$ and a cardinal number $\kappa$, we denote the free pro-$p$ group of rank $\kappa$ by $\widehat{F}_{\kappa}(p)$ (see \cite[Remark 17.4.7]{FrJa}).

We introduce below a notation for the most important profinite groups in this paper.
\begin{definition}\label{defkg}
Let $G$ be an arbitrary group and $\mathcal{G}$ be a profinite group.
\begin{enumerate}
\item We denote by
$$\mathcal{K}_{\mathcal{G}}:=\ker\left(\widetilde{\mathcal{G}}\ra \mathcal{G}\right)$$
the kernel of the universal Frattini cover of $\mathcal{G}$.

\item We also use the following notation:
$$\mathcal{K}_G:=\mathcal{K}_{\widehat{G}}.$$

\item We sometimes say that ``$\mathcal{G}$ is far from being projective'', if the profinite group $\mathcal{K}_{\mathcal{G}}$ is not small.
\end{enumerate}
\end{definition}
We give below a few examples of the notions mentioned above. For any $m>0$, we denote by $C_m$ the cyclic group of order $m$ written multiplicatively.
\begin{example}\label{example}
Let $p$ be a prime number and $n>0$.
\begin{enumerate}
\item We have:
$$\widehat{C_p}=C_p,\ \ \ \widetilde{C_p}=\Zz_p.$$
We obtain that:
$$\mathcal{K}_{C_p}=p\Zz_p\cong \Zz_p,$$
which is a small profinite group. One can show that in general finite groups are \emph{not} ``far from being projective''.

\item The profinite completion of the free group $\widehat{F_n}$ is a free profinite group, therefore it is projective. Hence, we obtain:
$$\widetilde{\widehat{F_n}}=\widehat{F_n},\ \ \ \ \mathcal{K}_{F_n}=\{1\}.$$

\item It is shown in \cite[Section 9]{Sjo} that
$$\mathcal{K}_{\Zz\times \Zz}\cong \prod_{p\text{:prime}}\widehat{F}_{\omega}(p),$$
therefore $\widehat{\Zz\times \Zz}$ is ``far from being projective''. We will generalize this result to the nilpotent case in Section \ref{sec4}.

\item Suppose that $G$ is an infinite, finitely generated, virtually free, and not free group. By \cite[Theorem 4.6]{BK}, the profinite group $\mathcal{K}_G$ is not small, so $G$  is ``far from being projective''.
\end{enumerate}
\end{example}
For a field $F$, we denote by $F^{\sep}$ a fixed separable closure of $F$, by $F^{\alg}$ a fixed algebraic closure of $F$, and by
$$\gal(F):=\gal(F^{\sep}/F)$$
the absolute Galois group of $F$ (considered as a profinite group). We would like to single out a standard result about actions on the absolute Galois groups.
\begin{fact}\label{galcont}
Let $K$ be a field and
$$\aut_{\{K\}}\left(K^{\sep}\right):=\{\sigma\in \aut\left(K^{\sep}\right)\ |\ \sigma(K)=K\}\leqslant \aut\left(K^{\sep}\right).$$
Then, $\gal(K)\trianglelefteqslant \aut_{\{K\}}(K^{\sep})$ and the conjugation induces an action of $\aut_{\{K\}}(K^{\sep})$ on $\gal(K)$ by continuous automorphisms.
\end{fact}
Let us fix a group $G$. By a \emph{$G$-field}, we mean a field together with an action of $G$ by field automorphisms (see \cite{HK3}, \cite{BK}, and  \cite{BK2}). A $G$-field $K$ is \emph{$G$-closed} if the action of $G$ does not extend to any proper algebraic extension of $K$. The class of $G$-fields coincides with the class of models of the obvious theory of $G$-fields in the language of rings extended by unary function symbols for the elements of $G$. We say that a $G$-field is \emph{existentially closed} (abbreviated \emph{e.c.}), if it is an existentially closed model of the theory of $G$-fields. If the class of e.c. $G$-fields is elementary, then we denote the theory of this class (a model companion of the theory of $G$-fields) by $G$-TCF and we say that ``$G$-TCF exists''. Otherwise, we say that ``$G$-TCF does not exist''.

We recall below some results about absolute Galois groups of $G$-fields. The statements in the theorem below originate from \cite{Sjo} (Theorems 4,5, and 6 in \cite{Sjo}). We use the formulations from \cite{BK2} (Lemma 2.7, Corollary 2.13, and Corollary 2.14 in \cite{BK2}), where the assumptions are a bit different and the proofs are more elaborate. We give below a counterexample to \cite[Theorem 6]{Sjo} (see Remark \ref{remabs}(3)).

\begin{theorem}[\cite{Sjo} and \cite{BK2}]\label{absgf}
Assume that the group $G$ is finitely generated. Let $K$ be an e.c. $G$-field and $C$ be the subfield of $G$-invariants.
\begin{enumerate}
\item We have:
  $$\gal(C^{\sep}/K\cap C^{\sep})\cong \widehat{G}.$$

\item We have:
$$\gal(C)\cong \widetilde{\widehat{G}}.$$

\item There is a natural continuous epimorphism (see Definition \ref{defkg}):
$$\gal(K)\twoheadrightarrow \mathcal{K}_G.$$
\end{enumerate}
\end{theorem}
\begin{remark}\label{remabs}
We would like to comment here on Theorem \ref{absgf}(3) and its relation to \cite[Theorem 6]{Sjo}.
\begin{enumerate}
\item It is claimed in \cite[Theorem 6]{Sjo} that after assuming that $G$ is finitely presented, the natural map $\gal(K)\twoheadrightarrow \mathcal{K}_G$
is an isomorphism. We show in Item (3) below that it need not be true. As it was discussed already in \cite[Remark 2.18(2)]{BK2}, this map is an isomorphism for a finite group $G$ and for a free group $G$.

\item We can partially confirm the statement from \cite[Theorem 6]{Sjo} for some particular groups, we give details in Remark \ref{c2c2}.

\item Let us consider the group $C_2\ast C_3=\langle a,b\rangle$ which is hyperbolic (this is folklore, see e.g. \cite[Proposition 3.2.A]{essgp} or \cite[Corollary 3]{khmy}). By \cite[Theorem 1]{ols}, the group $C_2\ast C_3$ has a finitely presented infinite quotient $G$ such that $\widehat{G}$ is trivial, hence $\mathcal{K}_G$ is trivial as well. We will see below that this group $G$ is a counterexample to \cite[Theorem 6]{Sjo}.

  Let $b'$ be the image of $b$ in $G$. Clearly, $b'$ still has order $3$, since otherwise $G$ would have order at most $2$. If \cite[Theorem 6]{Sjo} was true for this $G$, then we would obtain a faithful action of $G$ on an algebraically closed field (an e.c. $G$-field), contradicting the Artin-Schreier Theorem, since $b'$ would give an automorphism of order $3$ of an algebraically closed field.

\end{enumerate}
\end{remark}

\section{Groups with large universal Frattini kernels}\label{sec3}
In this section, we show the main result of this paper (Theorem \ref{main}) about non-existence of the theory $G$-TCF in the case when the profinite group $\mathcal{K}_G$ is not small (see Section \ref{sec2} for the necessary notions).

We note below an obvious result.
\begin{lemma}\label{3}
Let $\mathfrak{G}$ be a topological group, $H$ be a countable group acting on $\mathfrak{G}$ by continuous automorphisms and $A\subseteq \mathfrak{G}$ be a countable subset. Let us assume that $\mathfrak{G}$ is not topologically countably generated (that is: there is no countable subgroup of $\mathfrak{G}$, which is dense in $\mathfrak{G}$). Then, there is a subgroup $\mathcal{G}\leqslant \mathfrak{G}$ such that:
\begin{enumerate}
\item for each $h\in H$,  we have $h(\mathcal{G})=\mathcal{G}$;

\item $A\subseteq \mathcal{G}$;

\item $\mathcal{G}$ is closed and proper in $\mathfrak{G}$.
\end{enumerate}
\end{lemma}
\begin{proof}
It is straightforward to check that the following subgroup: 
$$\mathcal{G}:=\mathrm{cl}\left(\left\langle \bigcup_{h\in H} h(A)\right\rangle\right)$$ 
satisfies Items $(1)$--$(3)$ above.
\end{proof}
\begin{remark}
The above notion of ``topologically countably generated'' actually coincides with the notion of ``separable'' (having a dense countable subset), which also appears at the end of the proof of Theorem \ref{c2c2}.
\end{remark}
\begin{lemma}\label{4}
Let $G$ be a finitely generated group such that the profinite group $\mathcal{K}_G$ (see Definition \ref{defkg})
is not small and let $K$ be an e.c. $G$-field. Then $K$ is not bounded, that is there is $n>0$ such that $K$ has infinitely many extensions of degree $n$ (inside $K^{\sep}$).
\end{lemma}
\begin{proof}
Since any continuous quotient of a small profinite group is again small, the result follows directly from Theorem \ref{absgf}(3).
\end{proof}
\begin{remark}\label{reminf}
Let $F$ be a field and $n>0$. It can be easily checked that the following properties of $F$ are equivalent.
\begin{enumerate}
  \item The field $F$ has infinitely many extensions of degree $n$ in $F^{\sep}$.

  \item The field $F$ has infinitely many isomorphism classes of separable extensions of degree $n$.

  \item There are infinitely many separable irreducible polynomials $f_1,f_2,\ldots \in F[X]$ of degree $n$ such that for all $i\neq j$ and for all $\alpha,\beta\in F^{\sep}$, if $f_i(\alpha)=0=f_j(\beta)$ then $F(\alpha)\neq F(\beta)$.
\end{enumerate}
\end{remark}
\begin{lemma}\label{5}
Let $G$ be a countable group and $K$ be a $G$-field. Assume that $\gal(K)$ is not countably topologically generated. Then $K$ is not $G$-closed.
\end{lemma}
\begin{proof}
Let us choose a presentation:
$$G=\langle \tau_i:i<\omega\ |\ w_j(\bar{\tau}):j<\omega\rangle,$$
where $w_j$ are words and $\bar{\tau}=(\tau_i)_{i<\omega}$. We identify $G$ with a subgroup of $\mathrm{Aut}(K)$, since without loss of generality the action of $G$ on $K$ is faithful. For any $\sigma\in G$, we choose $\sigma'\in \mathrm{Aut}_{\{K\}}(K^{\sep})$ (see Fact \ref{galcont} for the notation) extending $\sigma$ and we define:
$$\bar{\tau}':=(\tau_i')_{i<\omega},\ \ \ A:=\{w_j(\bar{\tau}')\ :\ j<\omega\}\subseteq \mathrm{Aut}_{\{K\}}(K^{\sep}),\ \ \
H:= \langle \bar{\tau}' \rangle  \leqslant \mathrm{Aut}_{\{K\}}(K^{\sep}).$$
Since for each $j<\omega$, we have:
$$w_j(\bar{\tau}')|_K=w_j(\bar{\tau})=\id_K,$$
we obtain that $A\subseteq \gal(K)$. Therefore, using Fact \ref{galcont}, we can apply Lemma \ref{3} for $\mathfrak{G}=\gal(K)$ and $H,A$ as above. By Lemma \ref{3}, we obtain a closed proper subgroup $\mathcal{G}<\gal(K)$ such that  $A\subseteq \mathcal{G}$ and for each $h\in H$, we have $h\mathcal{G}h^{-1}=\mathcal{G}$. Let us take $M:=(K^{\sep})^{\mathcal{G}}$. Since for each $h\in H$, we have $h\mathcal{G}h^{-1}=\mathcal{G}$, we obtain that $H$ acts on $M$ by field automorphisms. Since $A\subseteq \mathcal{G}$, the above action of $H$ on $M$ yields an action of $G$ on $M$ extending the action of $G$ on $K$. Since $\mathcal{G}$ is a proper subgroup of $\gal(K)$, the algebraic extension $K\subseteq M$ is proper as well, therefore the $G$-field $K$ is not $G$-closed.
\end{proof}
\begin{remark}\label{c2c2}
Our proof of Lemma \ref{5} above has some similarities to Sj\"{o}gren's argument towards \cite[Theorem 6]{Sjo}, which does not hold in general (see Remark \ref{remabs}(3)). Using some additional properties of profinite groups and the ideas from the proof of Lemma \ref{5}, we can show the statement from \cite[Theorem 6]{Sjo} in some specific cases, for example if $G$ is of the form $C_p\ast \ldots \ast C_p$ for a prime $p$ (however, we probably get an abstract isomorphism, rather than showing that the natural epimorphism from Theorem \ref{absgf}(3) is an isomorphism).
\end{remark}
\begin{theorem}\label{main}
Let $G$ be a finitely generated group such that the profinite group $\mathcal{K}_G$
is not small. Then, the theory $G$-TCF does not exist.
\end{theorem}
\begin{proof}
By Lemma \ref{4}, any e.c $G$-field is not bounded. We assume that the theory $G$-TCF exists and we will reach a contradiction.
\\
\\
{\bf Claim}
\\
There is a model $K$ of the theory $G$-TCF such that $\gal(K)$ is not topologically countably generated.
\begin{proof}[Proof of Claim]
Let us take any e.c. $G$ field $M$ (so $M$ is a model of $G$-TCF). By Lemma \ref{4}, there is $n>0$ such that $M$ has infinitely many extensions of degree $n$ (inside $K^{\sep}$). For any field $F$ and any $\bar{a}=(a_0,\ldots,a_{n-1})\in F^n$, we set:
$$f_{\bar{a}}:=a_{0}+a_{1}X+\ldots+a_{n-1}X^{n-1}+X^n\in F[X].$$
By \cite[Section 3: (3.3), (3.4), and (3.5)]{ChaHel} and Remark \ref{reminf}, there is a formula $\varphi(\bar{x},\bar{y})$ in the language of rings, where $|\bar{x}|=|\bar{y}|=n$, such that for any field $F$ and any $\bar{a},\bar{b}\in F^n$, we have that $F\models \varphi(\bar{a},\bar{b})$ if and only if the following two conditions hold:
\begin{enumerate}
\item the polynomials $f_{\bar{a}},f_{\bar{b}}$ are irreducible in $F[X]$;

\item for any $\alpha,\beta\in F^{\sep}$, if $f_{\bar{a}}(\alpha)=0=f_{\bar{b}}(\beta)$ then $F(\alpha)\neq F(\beta)$.
\end{enumerate}
For any cardinal number $\kappa$, let us consider the language $L_{\kappa}$ which is the language of $G$-fields extended by
$\kappa$ many $n$-tuples $(\bar{c_i})_{i<\kappa}$ of constant symbols. Let $T_{\kappa}$ be the following $L_{\kappa}$-theory:
$$T_{\kappa}:= \text{$G$-$\mathrm{TCF}$} \cup \{\varphi(\bar{c}_i,\bar{c}_j)\ |\ i<j<\kappa\}.$$
By Compactness Theorem and Items $(1)$, $(2)$ above, the theory $T_{\kappa}$ is consistent. Therefore, there are models of $G$-TCF with arbitrarily large absolute Galois groups, in particular: there is a model $K$ of the theory $G$-TCF such that $\gal(K)$ is not topologically countably generated, since the cardinality of any topologically countably generated profinite group is bounded by $\beth_2=2^{2^{\aleph_0}}$ (see e.g. \cite[Exercise 3.5.14]{progps} or much more generally: the maximum possible cardinality of a separable Hausdorff space is $\beth_2$ as well).
\end{proof}
We take the $G$-field $K$ from Claim. By Lemma \ref{5}, $K$ is not $G$-closed, therefore $K$ is not e.c., a contradiction.
\end{proof}
The result below can be considered as a ``strong negation'' to \cite[Theorem 3.26]{BK}.
\begin{cor}\label{strneg}
Let $G$ be a finitely generated virtually free group. Then, the theory $G$-$\mathrm{TCF}$ exists if and only if $G$ is finite or $G$ is free.
\end{cor}
\begin{proof}
Since it is well-known that if $G$ is finite or $G$ is free, then the theory $G$-TCF exists (see \cite{HK3}, \cite{acfa1}, and \cite{BHKK}), it is enough to show the left-to-right implication. Suppose that $G$ is an infinite, finitely generated virtually free group, which is not free. By \cite[Theorem 4.6]{BK}, the profinite group $\mathcal{K}_G$ is not small. By Theorem \ref{main}, the theory $G$-TCF does not exist.
\end{proof}
\begin{remark}\label{remque}
We discuss here some additional issues related to the results above.
\begin{enumerate}
\item The assumption on finite generation of $G$ was not used directly in the arguments above, however, to be able to use the crucial Theorem \ref{absgf}, we need to assume that $G$ is finitely generated. It is still possible that Theorem \ref{main} is true without the assumption that $G$ is finitely generated, since we do not have a counterexample for such a more general statement.

\item Hoffmann conjectured (\cite[Conjecture 5.2]{Hoff3}) that if a theory $T$ has a model companion and a group $G$ is finite, then the theory of $G$-actions on models of $T$ has a model companion as well. If we take for $T$ the theory of difference fields (which are $\Zz$-fields in our terminology) and for $G$ any finite non-trivial group, then the theory of $G$-actions on models of $T$ is the same as the theory of $(\Zz\times G)$-fields. By Corollary \ref{strneg}, the theory $(\Zz\times G)$-TCF does not exists, so we get a counterexample to Hoffmann's conjecture.

\item It was also conjectured (private communications) that if $G$ and $H$ are groups and the theories $G$-TCF and $H$-TCF exist, then the theory $(G\ast H)$-TCF exists as well. Using Corollary \ref{strneg}, we can see that this not the case, for example if one takes $G=H=C_2$.

\end{enumerate}

\end{remark}

\section{Nilpotent groups and summary}\label{sec4}
In this section, we give the full description of those finitely generated nilpotent groups $G$ for which the theory $G$-TCF exists and we also summarize what we know about the companionability of the theories of group actions on fields.

\subsection{Finitely generated nilpotent groups}\label{sec41}
We will use the fact that pronilpotent groups are fully described by their pro-$p$ Sylow subgroups (see Section \ref{sec2}). The crucial preparatory result is the following, which may be a folklore.
\begin{prop}\label{rk2prop}
Assume $N$ is an infinite finitely generated nilpotent group which is not cyclic. Then, there is a prime number $p$ such that $\widehat{N}_p$ is infinite and $\mathrm{rk}(\widehat{N}_p)\geqslant 2$.
\end{prop}
\begin{proof}
 It is enough to find a quotient of $N$ for which the result holds. We will often use a fact saying that if a group $H$ acts on a finitely generated group $G$, then we have:
 $$\widehat{G} \rtimes \widehat{H}\cong \widehat{G\rtimes H}.$$
 Since $N$ is supersolvable, $N$ has an infinite virtually cyclic quotient $C$. By \cite[Lemma 11.4]{3m}, $C$ is of the form $G\rtimes \Zz$ where $G$ is finite and nilpotent or $C$ projects onto $\Zz\rtimes C_2$. Since we have:
 $$\left(\widehat{\Zz\rtimes C_2}\right)_2\cong \Zz_2\rtimes C_2,$$
 we can assume that $C=G\rtimes \Zz$. If $G$ is non-trivial, then we take a prime number $p$ dividing the order of $G$ and obtain:
 $$\left(\widehat{G\rtimes \Zz}\right)_p\cong G_p\rtimes \Zz_p.$$
The above pro-$p$ group is infinite and rank at least 2.

Therefore, we can assume that $C=\Zz$, so $N\cong N_0\rtimes \Zz$, where $N_0$ is a non-trivial finitely generated nilpotent group, since subgroups of finitely generated nilpotent groups are again finitely generated. In this case, we can proceed as above taking a prime $p$ which divides the order of a non-trivial finite quotient of $N_0$.
\end{proof}
We obtain the following.
\begin{theorem}\label{nilpfar}
 Assume $N$ is an infinite, finitely generated nilpotent group which is not cyclic. Then, the profinite group $\mathcal{K}_N$ is not small.
\end{theorem}
\begin{proof}
By Proposition \ref{rk2prop}, there is a prime  number $p$ such that $\mathrm{rk}(\widehat{N}_p)\geqslant 2$. Since for any pronilpotent group $\mathcal{N}$, we have:
$$(\widetilde{\mathcal{N}})_p=\widetilde{(\mathcal{N}_p)},$$
it is enough to show that $\mathcal{K}_{\widehat{N}_p}$ is not small. However, since a pro-$p$ group is projective if and only if it is pro-$p$ free (this is a result of Tate, see \cite[Proposition 22.7.6]{FrJa}), we get that:
$$\widetilde{\widehat{N}}_p\cong \widehat{F}_{r}(p),$$
where $r=\mathrm{rk}(\widehat{N}_p)\geqslant 2$. Since $\widehat{N}_p$ is not only pronilpotent but also nilpotent, we get that the universal Frattini cover map
$$\widetilde{\widehat{N}}_p\ra \widehat{N}_p$$
is not an isomorphism ($\widehat{F}_{r}(p)$ is not nilpotent for $r\geqslant 2$, since it contains a free group on two generators as a subgroup, see \cite[Proposition 3.3.6]{progps}). Therefore, $\mathcal{K}_{\widehat{N}_p}$ is a closed, normal, non-trivial, and infinite index subgroup of $\widehat{F}_{r}(p)$. By \cite[Proposition 8.6.3]{progps}, we obtain:
$$\mathcal{K}_{\widehat{N}_p}\cong \widehat{F}_{\omega}(p),$$
so $\mathcal{K}_{\widehat{N}_p}$ is not small.
\end{proof}
\begin{cor}\label{nilpans}
Let $N$ be a finitely generated nilpotent group. Then, the theory $N$-$\mathrm{TCF}$ exists if and only if $N$ is finite or $N$ is cyclic.
\end{cor}
\begin{proof}
The proof is the same as the proof of Corollary \ref{strneg}, where we replace \cite[Theorem 4.6]{BK} with Theorem \ref{nilpfar}
\end{proof}
Since the group $\Zz\times \Zz$ is infinite, nilpotent, and not cyclic, Corollary \ref{nilpans} includes Hrushovski's result about the non-companionability of the theory of fields with two commuting automorphisms.

\subsection{Free products}\label{secfree}
This part is inspired by a question of Alexander Ivanov regarding the existence of the theory $(\Zz^2\ast \Zz)$-TCF. The motivation for this question comes from the fact that the group $\Zz^2\ast \Zz$ is fully residually free and we know that for a free group $F$, the theory $F$-TCF exists.

To answer the question above, we need the following general result.
\begin{theorem}\label{thmfree}
For any groups $G,H$, we have a natural epimorphism:
$$\mathcal{K}_{G\ast H}\ra \mathcal{K}_{G}.$$
\end{theorem}
\begin{proof}
By \cite[Exercise 9.1.1(a) and Corollary 9.1.4(a)]{progps}, we get that $\widehat{G}$ is topologically isomorphic with a closed subgroup of $\widehat{G\ast H}$. By \cite[Lemma 4.3]{BK}, there is a continuous epimorphism:
$$\mathcal{K}_{\widehat{G\ast H}}\ra \mathcal{K}_{\widehat{G}},$$
which gives the result.
 \end{proof}
\begin{cor}
For any finitely generated group $H$, the theory $(\Zz^2\ast H)$-$\mathrm{TCF}$ does not exist.
\end{cor}
\begin{proof}
By Example \ref{example}(3), we have that the profinite group $\mathcal{K}_{\Zz^2}$ is not small. By Theorem \ref{thmfree}, we obtain that the profinite group $\mathcal{K}_{\Zz^2\ast H}$ is not small either. By Theorem \ref{main}, the theory $(\Zz^2\ast H)$-TCF does not exists.
\end{proof}

\subsection{Summary}\label{secsum}
In this final section, we give a summary of what we know regarding existence of the theories $G$-TCF for different types of groups $G$.

If the group $G$ is finitely generated, then we do not know any ``new'' (that is: infinite and not free) types of groups $G$ such that the theory $G$-TCF exists. Therefore, it is reasonable to ask the following.
\begin{question}\label{q1}
Suppose that $G$ is an infinite and finitely generated group. Does the theory $G$-TCF exists if and only if $G$ is free?
\end{question}
Corollary \ref{nilpans} answers Question \ref{q1} positively in the case of a nilpotent group $G$. Possibly, the methods of Section \ref{sec41} could be extended to supersolvable or even solvable groups. On a rather ``orthogonal'' side of the spectrum of finitely generated groups, Question \ref{q1} has an affirmative answer for virtually free groups by Corollary \ref{strneg}.
\begin{remark}
It is natural to start with checking for which finitely generated groups $G$, the profinite group $\mathcal{K}_G$ is not small. We know that:
\begin{enumerate}
  \item if $G$ is finite or free, then $\mathcal{K}_G$ is small;
  \item excluding Item (1) above, $\mathcal{K}_G$ is not small in the case when $G$ is virtually free \cite[Theorem 4.6.]{BK};
  \item $\mathcal{K}_G$ is small (even trivial, since $\widehat{G}$ is trivial!) for very complicated groups like \emph{Tarski monster groups} (see e.g. \cite[Section 1]{sapir}) or the \emph{Higman group} (see \cite{high}).
\end{enumerate}
\end{remark}
We do not have any counterexample for the following.
\begin{question}\label{sbgp}
Suppose that $H<G$ and the theory $G$-TCF exists. Does the theory $H$-TCF exist as well?
\end{question}
We conjectured \cite[Conjecture 6.6]{BK2} that Question \ref{sbgp} has the affirmative answer for $H=\Zz\times \Zz$, but we were able to show only a slightly weaker (and a bit surprising) result, which is \cite[Corollary 6.9]{BK2}.
\begin{remark}\label{pecvsgc}
We would like to point out here an important difference between Hrushovski's proof of non-existence of the theory $(\Zz\times \Zz)$-TCF and our proof of a more general result (Corollary \ref{nilpans}). Hrushovski focused on \emph{p.e.c.} (\emph{pseudo e.c.}) $G$-fields, that is $G$-fields which are existentially closed in those $G$-field extensions, which are regular extensions of pure fields. It is rather clear that a $G$-field is e.c. if and only if it is p.e.c. and $G$-closed (see \cite[Remark 2.3(1)]{BK2}). Hrushovski's proof gives that actually there are no saturated p.e.c. $(\Zz\times \Zz)$-fields (see \cite[Theorem 6.7]{BK2}), which is a stronger statement comparing to our result (in the case of $G=\Zz\times \Zz$).
\end{remark}
The following question is related to Remark \ref{pecvsgc} above.
\begin{question}\label{q2}
Suppose that $G$ is a finitely generated and virtually free group. Is the class of p.e.c. $G$-fields elementary?
\end{question}
The methods from \cite{BK} could be used to attack Question \ref{q2}, but we do not know how to do it exactly.

A positive answer to Question \ref{q1} would be quite negative for the whole theory, since there will be no ``new'' theories in the case of finitely generated groups. Therefore, one can turn the attention to the arbitrary groups. The first non-free infinite case was considered by Medvedev \cite{med1} who showed that (in our terminology) the theory $\Qq$-TCF exists. In \cite{BK2}, we obtained the full classification of torsion Abelian groups $A$ such that the theory $A$-TCF exists. The next step could be to extend this classification to arbitrary Abelian groups. Having in mind the results of this paper, the following may be reasonable.
\begin{question}\label{q3}
Suppose that $A$ is an Abelian group. Is it true that the theory $A$-TCF exists if and only if none of the following groups embed in $G$ ($p$ is a prime number below)?
\begin{itemize}
  \item $\Zz\times \Zz$.
  \item $\Zz\times C_p$.
  \item $C_p^{(\omega)}$, which is the infinite countable direct sum of $C_p$'s.
  \item $C_p\times C_{p^{\infty}}$, where $C_{p^{\infty}}$ is the Pr\"{u}fer $p$-group.
\end{itemize}
\end{question}
The last two types of groups are the ``forbidden groups'' from \cite[Remark 1.2(1)]{BK2}. The results from \cite{BK2} can be possibly extended to locally finite nilpotent groups and then one could appropriately generalize Question \ref{q3} above.
\bibliographystyle{plain}
\bibliography{harvard}

\end{document}